\documentclass[a4paper,10pt]{article}
\usepackage{mystyle}
\title{Twisted Coefficients on coarse Spaces and their Corona}
\author{Elisa Hartmann}

\begin{document}

\maketitle

\begin{abstract}
To a metric space $X$ we associate a compact topological space $\corona X$ called the corona of $X$. Then a coarse map $f:X\to Y$ between metric spaces is mapped to a continuous map $\corona f:\corona X\to \corona Y$ between coronas. Sheaf cohomology on coarse spaces has been introduced in \cite{Hartmann2017a}. We show the functor $\nu'$ preserves and reflects sheaf cohomology.
\end{abstract}

\tableofcontents

\section{Introduction}

The corona of a proper metric space has been introduced in \cite{Protasov2003}. We show it suggests a duality between the coarse structure of a metric space and the topology of its boundary.

\subsection{Background and related Theories}

There are a number of dualities and natural equivalences between categories that relate different areas of mathematics.

The Gelfand duality relates topological spaces with their algebra of functions. Note that \cite{Brandenburg2015} gives a very concise introduction to Gelfand duality. Denote by $\locktop$ the category of locally compact Hausdorff spaces and proper\footnote{Note that a continuous map is \emph{proper} if the inverse image of any compact space is compact} continuous maps. Likewise we denote by $\commcstar$ the category of commutative $C^*$-algebras and nondegenerate $*-$homomorphisms. Then the gelfand representation 
 \[
  \gamma:\commcstar^{op}\to\locktop
 \]
is a fully faithful functor.

According to \cite[Chapter~1.3]{Khalkhali2009} Hilbert's Nullstellensatz implies there is an equivalence between the category of affine algebraic varieties over an algebraically closed field $F$ and the the opposite of the category of finitely generated commutative reduced unital $F$-algebras.

We introduce another duality in the realm of mathematics. Coarse geometry of metric spaces is related to the topology of compact spaces by designing yet another version of boundary. The new corona functor may not be an equivalence of categories, it is the closest we can get though.

In \cite{Kalantari2016},\cite{Kalantari2015} has been designed and studied a new relation on subsets of a coarse space: Two subsets $A,B\s X$ of a metric space are related by $\delta_\lambda$, written $A\delta_\lambda B$ if $A\close B$ or $A\cap B\not=\emptyset$. We know the close relation $\close$ from \cite{Hartmann2017b}. Now $\delta_\lambda$ is a proximity relation on subsets of a asymptotically normal\footnote{The notion \emph{asymptotically normal} has been defined in \cite{Kalantari2016}.} coarse spaces. The boundary of the Smirnov compactification associated to the proximity space $(X,\delta_\lambda)$ has been proven in \cite{Kalantari2016} to be homeomorphic to the Higson corona $\nu X$ if $X$ is a proper metric space.

In \cite{Protasov2015} has been associated to a metric space $X$ a corona $\check X$ as a quotient of the Stone-\v Chech compactification. The boundary has been studied in various papers \cite{Protasov2003},\cite{Protasov2005},\cite{Protasov2011},\cite{Banakh2013},\cite{Protasov2015}. Again if $X$ is a proper metric space then $\check X$ is homeomorphic to the Higson corona $\nu(X)$. 

This paper introduces a boundary on coarse metric spaces which is a functor that maps coarse maps to continuous maps between compact spaces. We show it suggests a duality since sheaf cohomology on coarse spaces is isomorphic to the sheaf cohomology of its boundary. In designing this functor we looked specifically for this property. We obtain a boundary which is yet again homeomorhic to the Higson corona if $X$ is proper.

If $X$ is a coarsely proper metric space all three definitions of corona, the one introduced in \cite{Kalantari2016} called $\gamma X/\approx$, the one in \cite{Protasov2015} called $\check X$ and in the one in this paper called $\corona X$ can be proven to be equivalent. 

\subsection{Main Contributions}
To every proximity space we can associate a compact space which arises as the boundary of the Smirnov compactification. The close relation $\close$ on subsets of a metric space is almost but not quite a proximity relation. We can still mirror the construction of the Smirnov compactification as described in \cite{Naimpally1970}.

The Definition~\ref{defn:cu} of coarse ultrafilters and Definition~\ref{defn:asymptoticallyalikecu} of the relation asymptotically alike on coarse ultrafilters combine to a coarse version of a cluster on $(X,\close)$. Note that we do not change much. In Lemma~\ref{lem:cucoarsemap}, Lemma~\ref{lem:closesamemap} we show that coarse maps preserve coarse ultrafilters modulo asymptotically alike.

In Definition~\ref{defn:topologyoncu} we define a topology on coarse ultrafilters modulo asymptotically alike. We call the resulting space the corona $\corona X$ of a metric space $X$. Note again this mirrors the topology on clusters in a Smirnov compactification.

Then Lemma~\ref{lem:fE(f)} shows that $\nu'$ is a functor:
\begin{thma}
Denote by $\mcoarse$ the category with metric spaces as objects and coarse maps modulo close as morphisms. By $\topology$ denote the category of topological spaces and continuous maps. Then 
\[
\nu':\mcoarse\to \topology
\]
is a functor. If $X$ is a metric space then $\corona X$ is compact.
\end{thma}

Note we are not able to show that $\corona X$ is metrizable in general. In fact Remark~\ref{rem:metrizable} and Proposition~\ref{prop:countablebase} suggest the opposite.

The Lemma~\ref{lem:coarsecovertopology} studies the topology of the corona $\corona X$ of a metric space $X$ and shows that open covers of $\corona X$ can be refined by open covers that are induced by coarse covers on $X$. The Corollary~\ref{cor:cor1}, Corollary~\ref{cor:cor2}, Corollary~\ref{cor:cor3}, Corollary~\ref{cor:cor4}, Corollary~\ref{cor:cor5} show this is a powerful property. Among them is:    

\begin{thma}
Let $X$ be a metric space. If $\sheaff$ is a sheaf on $X$ then
\[
\cohomology i X \sheaff = \cohomology i {\corona X} \sheaff
\]
here the left side denotes sheaf cohomology on coarse spaces and the right side denotes sheaf cohomology on a topological space. Likewise if $\sheaff$ is a sheaf on $\corona X$ the same statement holds. 
\end{thma}

The Lemma~\ref{lem:closedembedding} shows that every coarsely injective map induces a closed embedding between coronas. Conversely Theorem~\ref{thm:surjectivecoarselysurjective} shows that $\nu'$ reflects epimorphisms.

\section{Metric Spaces}
\label{sec:metric}
\begin{defn}
 Let $(X,d)$ be a metric space. Then the \emph{coarse structure associated to $d$} on $X$ consists of those subsets $E\s X^2$ for which
 \[
  \sup_{(x,y)\in E}d(x,y)<\infty.
 \]
 We call an element of the coarse structure \emph{entourage}.  In what follows we assume the metric $d$ to be finite for every $(x,y)\in X^2$.
\end{defn}

\begin{defn}
 If $X$ is a metric space a subset $B\s X$ is \emph{bounded} if the set $B^2$ is an entourage in $X$.
\end{defn}

\begin{defn}
 A map $f:X\to Y$ between metric spaces is called \emph{coarse} if
 \begin{itemize}
  \item $E\s X^2$ being an entourage implies that $\zzp f E$ is an entourage \emph{(coarsely uniform)};
  \item and if $A\s Y$ is bounded then $\iip f A$ is bounded \emph{(coarsely proper)}.
 \end{itemize}
 Two maps $f,g:X\to Y$ between metric spaces are called \emph{close} if
 \[
 f\times g(\Delta_X)
 \]
 is an entourage in $Y$. Here $\Delta_X$ denotes the diagonal in $X$.
\end{defn}

\begin{notat}
 A map $f:X\to Y$ between metric spaces is called
 \begin{itemize}
 \item \emph{coarsely surjective} if there is an entourage $E\s Y^2$ such that 
 \[
  E[\im f]=Y;
 \]
 \item \emph{coarsely injective} if for every entourage $F\s Y^2$ the set $\izp f F$ is an entourage in $X$.
  \item two subsets $A,B\s X$ are called \emph{coarsely disjoint} if for every entourage $E\s X^2$ the set
  \[
  E[A]\cap E[B]
  \]
  is bounded.
\end{itemize}
\end{notat}

\begin{rem}
 We study metric spaces up to coarse equivalence. A coarse map $f:X\to Y$ is a \emph{coarse equivalence} if
 \begin{itemize}
  \item There is a coarse map $g:Y\to X$ such that $f\circ g$ is close to $id_Y$ and $g\circ f$ is close to $id_X$.
 \item or equivalently if $f$ is both coarsely injective and coarsely surjective.
 \end{itemize}
\end{rem}

This is~\cite[Definition~3.D.10]{Cornulier2016}:
\begin{defn}\name{coarsely proper}
 If $X$ is a metric space we write
 \[
  B(p,r)=\{x\in X:d(x,p)\le r\}
 \]
 for a point $p\in X$ and $r\ge 0$. The space $X$ is called \emph{coarsely proper} if there is some $R_0>0$ such that for every bounded subset $B\s X$ the cover
\[
 \bigcup_{x\in B}B(x,R_0)
\]
of $B$ has a finite subcover.
\end{defn}

\section{The close Relation and coarse Covers}

This is \cite[Definition~9]{Hartmann2017b}.
\begin{defn}\name{close relation}
\label{defn:closerelation}
 Let $X$ be a coarse space. Two subsets $A,B\s X$ are called \emph{close} if they are not coarsely disjoint. We write 
 \[
  A\close B.
 \]
Then $\close$ is a relation on the subsets of $X$.
\end{defn}

\begin{lem}
\label{lem:close}
 In every metric space $X$:
 \begin{enumerate}
  \item if $B$ is bounded, $B\notclose A$ for every $A\s X$
  \item $U\close V$ implies $V\close U$
  \item $U\close (V\cup W)$ if and only if $U\close V$ or $U\close W$
  \item for every subspaces $A,B\s X$ with $A\notclose B$ there are subsets $C,D\s X$ such that $C\cap D=\emptyset$ and $A\notclose (X\ohne C)$, $B\notclose X\ohne D$.
 \end{enumerate}
\end{lem}
\begin{proof}
This is \cite[Lemma~10, Proposition~11]{Hartmann2017b}.
\end{proof}

\begin{rem}
 \label{rem:coarsecontinuous}
 If $f:X\to Y$ is a coarse map then whenever $A\close B$ in $X$ then $f(A)\close f(B)$ in $Y$.
\end{rem}

We recall~\cite[Definition~45]{Hartmann2017a}:
\begin{defn}\name{coarse cover}
 If $X$ is a metric space and $U\s X$ a subset a finite family of subsets $U_1,\ldots,U_n\s U$ is said to \emph{coarsely cover} $U$ if for every entourage $E\s X^2$ there exists a bounded set $B\s X$ such that 
\[
 U^2\cap (\bigcup_i U_i^2)^c\cap E\s B^2.
\]
\end{defn}

\begin{rem}
 Note that coarse covers determine a Grothendieck topology on $X$. If $f:X\to Y$ is a coarse map between metric spaces and $(V_i)_i$ a coarse cover of $V\s Y$ then $(\iip f{V_i})_i$ is a coarse cover of $\iip f V\s X$.
\end{rem}

\begin{lem}
\label{lem:pcover}
 Let $X$ be a metric space. A finite family $\ucover=\{U_\alpha:\alpha\in A\}$ is a coarse cover if and only if there is a finite cover $\vcover=\{V_\alpha:\alpha\in A\}$ of $X$ as a set such that $V_\alpha\notclose U_\alpha^c$ for every $\alpha$. 
\end{lem}
\begin{proof}
This is \cite[Lemma~16]{Hartmann2017b}.
\end{proof}

Recall \cite[Definition~2.1]{Kalantari2016}:
\begin{defn}
Two subsets of a metric space $S,T\s X$ are called \emph{asymptotically alike} if there is an entourage $E\s X^2$ such that $E[S]=T$. We write $S\lambda T$ in this case.
\end{defn}

\section{Coarse Ultrafilters}

\begin{defn}
\label{defn:cu}
If $X$ is a metric space a system $\sheaff$ of subsets of $X$ is called a \emph{coarse ultrafilter} if
\begin{enumerate}
\item $A,B\in \sheaff$ then $A\close B$.
\item $A,B\s X$ are subsets with $A\cup B\in\sheaff$ then $A\in\sheaff$ or $B\in\sheaff$.
\item $X\in\sheaff$.
\end{enumerate}
\end{defn}

\begin{lem}
\label{lem:cubasicprops}
If $X$ is a metric space and $\sheaff$ a coarse ultrafilter on $X$ and if $A\not\in \sheaff$ then $A^c\in \sheaff$.
\end{lem}
\begin{proof}
 Assume the opposite, both $A,A^c\not\in\sheaff$. Then $X=A\cup A^c\not\in \sheaff$ which is a contradiction to axiom 3.
\end{proof}

\begin{lem}
\label{lem:cucoarsemap}
If $f:X\to Y$ is a coarse map between metric spaces and $\sheaff$ is a coarse ultrafilter on $X$ then
\[
f_*\sheaff:=\{A\s Y:\iip f A\in\sheaff\}
\]
is a coarse ultrafilter on $Y$.
\end{lem}
\begin{proof}
\begin{enumerate}
\item If $A,B\in f_*\sheaff$ then $\iip f A,\iip f B\in\sheaff$. This implies $\iip f A\close \iip f B$ which implies $A\close B$. 
\item If $A,B\s Y$ are subsets with $A, B\not\in f_*\sheaff$ then $\iip f A, \iip f B\not\in\sheaff$. Which implies $\iip f A\cup\iip f B\not\in\sheaff$. Thus $A\cup B\not\in\sheaff$.
\item $\iip f Y=X$. Thus $Y\in f_*\sheaff$.
\end{enumerate}
\end{proof}

\begin{thm}
\label{thm:cuexists}
If $X$ is a coarsely proper metric space and $Z\s X$ an unbounded subset then there is a coarse ultrafilter $\sheaff$ on $X$ with $Z\in\sheaff$.
\end{thm}
\begin{proof}
We just need to prove there is a coarse ultrafilter $\sheaff$ on $Z$. Then $i_*\sheaff$ where $i:Z\to X$ is the inclusion has the required properties.

By Proposition~\cite[Lemma~88]{Hartmann2017a} we can assume that $Z$ is $R$-discrete for some $R>0$. Then the bounded sets are exactly the finite sets. The rest of the proof is very similar to the proof of \cite[Theorem~5.8]{Naimpally1970}. Let $\sigma$ be a non-principal ultrafilter on $Z$ (Thus every $A\in\sigma$ is not finite). We define
\[
\sheaff:=\{A\s Z:A\close C\mbox{ for each }C\in\sigma\}
\]
We check that $\sheaff$ is a coarse ultrafilter on $Z$:
\begin{enumerate}
\item If $A,B\in\sheaff$ let $C\s X$ be a subset. Then either $C\in\sigma$ or $C^c\in\sigma$. This implies both $A\close C,B\close C$ or both $A\close C^c,B\close C^c$. Thus for every $C\s Z$ we have $C\close A$ or $C^c\close B$ this implies $A\close B$.
\item If $A,B\s Z$ are subsets with $A,B\not\in\sheaff$ then there are $C_1,C_2\in\sigma$ with $A\notclose C_1,B\notclose C_2$. Then
\[
A\cup B\notclose C_1\cap C_2.
\]
Since $C_1\cap C_2\in\sigma$ we have $A\cup B\not\in\sheaff$.
\item $Z\in \sheaff$ since $Z\close A$ for every nonbounded subset $A\s Z$.
\end{enumerate}
\end{proof}

\begin{defn}
\label{defn:asymptoticallyalikecu}
We define a relation on coarse ultrafilters on $X$: two coarse ultrafilters $\sheaff,\sheafg$ are \emph{asymptotically alike}, written $A\lambda B$ if for every $A\in \sheaff,B\in \sheafg$:
\[
A\close B
\]
\end{defn}
\begin{lem}
The relation asymptotically alike is an equivalence relation on coarse ultrafilters on $X$.
\end{lem}
\begin{proof}
The relation is obviously symmetric and reflexiv. We show transitivity. Let $\sheaff_1,\sheaff_2,\sheaff_3$ be coarse ultrafilters on $X$ such that $\sheaff_1\lambda \sheaff_2$ and $\sheaff_2\lambda \sheaff_3$. We show $\sheaff_1\lambda \sheaff_3$. Assume the opposite. There are $A\in \sheaff_1$ and $B\in \sheaff_3$ such that $A\notclose B$. Then there are subsets $C,D\s X$ with $C^c\cup D^c=X$ and $C^c\notclose A,D^c\notclose B$. Now one of $C^c,D^c$ is in $\sheaff_2$. This contradicts $\sheaff_1\lambda\sheaff_2$ and $\sheaff_2\lambda \sheaff_3$. Thus transitivity follows.
\end{proof}

\begin{lem}
\label{lem:closesamemap}
If two coarse maps $f,g:X\to Y$ between metric spaces are close and $\sheaff,\sheafg$ are asymptotically alike coarse ultrafilters on $X$ then $f_*\sheaff\lambda g_*\sheafg$ in $Y$.
\end{lem}
\begin{proof}
If $A\in f_*\sheaff,B\in g_*\sheafg$ then $\iip f A\in\sheaff,\iip g B\in\sheafg$. This implies $\iip f A \close \iip g B$. Thus there are subsets $S\s \iip f A,T\s \iip g B$ which are not bounded such that $S\lambda T$. Since $f,g$ are close we have $f(S)\lambda g(T)$. Now $f(S)\s A,g(T)\s B$ are not bounded since $f,g$ are coarsely proper. This implies $A\close B$. Thus $f_*\sheaff\lambda g_*\sheaff$ in $Y$.
\end{proof}

\begin{prop}
\label{prop:alikecu}
Let $\sheaff,\sheafg$ be two coarse ultrafilters on a metric space $X$. Then $\sheaff\lambda \sheafg$ if and only if for every $A\in\sheaff$ there is an element $B\in\sheafg$ with $A\lambda B$.
\end{prop}
\begin{proof}
This has already been proved in \cite[Lemma~4.2]{Protasov2003}.
\end{proof}

\section{Topological Properties}

\begin{defn}
\label{defn:topologyoncu}
Let $X$ be a metric space. Denote by $\hat \nu (X)$ the set of coarse ultrafilters on $X$. We define a relation $\close$ on subsets of $\hat \nu(X)$ as follows. Define for a subset $A\s X$:
\[
\closedop A=\{\sheaff\in\hat \nu(X):A\in \sheaff\}
\]
Then $\pi_1\notclose \pi_2$ if and only if there exist subsets $A,B\s X$ such that $A\notclose B$ and $\pi_1\s \closedop A,\pi_2\s\closedop B$. 
\end{defn}

\begin{thm}
The relation $\close$ on $\hat \nu( X)$ is a proximity relation. If $\sheaff\in \hat \nu( X)$ Then
\[
\bar \sheaff=\{\sheafg\in \hat \nu( X):\sheafg\lambda \sheaff\}
\]
Thus $\close $ is a separated proximity relation on the quotient $\corona X:=\hat \nu( X)/\lambda$. We call $\corona X$ with the induced topology of $\close$ the \emph{corona of $X$}. 
\end{thm}
\begin{proof}
We check the axioms of a proximity relation:
\begin{enumerate}
\item if $\pi_1\close \pi_2$ then $\pi_2\close\pi_1$ since the definition is symmetric in $\pi_1,\pi_2$.
\item If $\pi_1=\emptyset$ then $\pi_1\s \closedop\emptyset$. Now for every $\pi_2\s \corona X$ we have $\pi_2\s\closedop X$ and $\emptyset\notclose X$. Thus $\pi_1\notclose \pi_2$.
\item If $\pi_1\cap \pi_2\not=\emptyset$ and if $\pi_1\s \closedop A, \pi_2\s\closedop B$ then there is some coarse ultrafilter $\sheaff\in\pi_1\cap\pi_2$, thus $A\in\sheaff$ and $B\in\sheaff$. This implies $A\close B$. Thus we have shown that $\pi_1\close\pi_2$.
\item Let $\pi_1,\pi_2\s \corona X$ be subsets such that for every $\rho\s \corona X$ we have $\pi_1\close \rho$ or $\pi_2\close \rho^c$. Let $A,B\s X$ be subsets such that $\pi_1\s \closedop A,\pi_2\s \closedop B$. Let $C\s X$ be a subset. 
Now one of $\pi_1\close \closedop C$ or $\pi_2\close \closedop C^c$ holds. We have $\pi_1\s \closedop A,\closedop C\s \closedop C$ and $\pi_2\s\closedop B,\closedop C^c\s \closedop {C^c}$. Thus $A\close C$ or $B\close C^c$. Now $C\s X$ was arbitrary thus $A\close B$. We have shown $\pi_1\close \pi_2$.
\item Let $\sheaff,\sheafg\in \corona X$ be two coarse ultrafilters such that $\sheaff\bar\lambda\sheafg$. Then there exist subsets $A,B\s X$ such that $A\notclose B$ and $A\in\sheaff,B\in\sheafg$. Thus $\sheaff\notclose \sheafg$.
\item if $\sheaff,\sheafg\in \corona X$ are two coarse ultrafilters such that $\sheaff\lambda \sheafg$ then for every $A\in\sheaff,B\in\sheafg$ we have $A\close B$. This implies $\sheaff\close \sheafg$.
\end{enumerate}
\end{proof}

\begin{lem}
\label{lem:fE(f)}
If $f:X\to Y$ is a coarse map between metric spaces then $f$ induces a proximity map $\corona f :\corona X\to \corona Y$.
\end{lem}
\begin{proof}
Define for a subset $S\s \corona X$:
\[
f_*S=\{[f_*\sheaff]:[\sheaff]\in S\}
\]
Let $\pi_1,\pi_2\s \corona X$ be two subsets with $f_*\pi_1\notclose f_*\pi_2$ in $\corona Y$. Then there are subsets $A,B\s Y$ such that $A\notclose B$ and $f_*\pi_1\s \closedop A,f_*\pi_2\s \closedop B$. Then $\iip f A\notclose \iip f B$ and $\pi_1\s \closedop{\iip f A},\pi_2\s\closedop{\iip f B}$. This implies $\pi_1\notclose \pi_2$.
\end{proof}

\begin{cor}
Denote by $\mcoarse$ the category of metric spaces and coarse maps modulo close. By $\proximity$ denote the category of proximity spaces and p-maps. Then
\[
\nu':\mcoarse\to \proximity
\]
is a functor.
\end{cor}

\begin{thm}
\label{thm:compact}
If $X$ is a metric space then the space $\corona X$ is compact.
\end{thm}
\begin{proof}
Let $(A_i)_i$ be a family of arbitrary small closed sets in $\corona X$ with the finite intersection property. We need to show $\bigcap_i A_i\not=\emptyset$. It is sufficient to show this property for $A_i=\closedop{B_i}$ for every $i$ where $B_i\s X$ are subsets for every $i$. We can assume $(\closedop{B_i})_i$ is maximal among families of closed subsets of $\corona X$ of the form $(\closedop{B_i})_i$ with the finite intersection property. Define
\[
\sheaff:=\{B_i:i\}.
\]
Then $\sheaff$ is a coarse ultrafilter:
\begin{enumerate}
\item if $B_i,B_j\in\sheaff$ and $B_i\notclose B_j$ then $\closedop{B_i}\cap\closedop{B_j}=\emptyset$ which is a contradiction to the finite intersection property.
\item If $A,B\not\in\sheaff$ then there are $B_1,\ldots,B_n\in \sheaff$ with
\[
\closedop{B_1}\cap\cdots\cap\closedop{B_n}\cap\closedop A=\emptyset.
\]
And there are $C_1,\ldots,C_n\in\sheaff$ with
\[
\closedop{C_1}\cap\cdots\cap\closedop{C_m}\cap\closedop B=\emptyset.
\]
Assume the opposite, $A\cup B\in \sheaff$. Since $\closedop A\cup \closedop {B}=\closedop{A\cup B}$ we get
\f{
\closedop{A\cup B}\cap\closedop{B_1}&\cap\cdots\cap\closedop{B_n}\cap\closedop{C_1}\cap\cdots\cap\closedop{C_m}\\
&=\closedop{A\cup B}\cap( (\closedop{B_1}\cap\cdots\cap\closedop{B_n}\cap\closedop{C_1}\cap\cdots\cap\closedop{C_m}\cap\closedop A)\\
&\quad\cup (\closedop{B_1}\cap\cdots\cap\closedop{B_n}\cap\closedop{C_1}\cap\cdots\cap\closedop{C_m}\cap\closedop{B}))\\
&\s (\closedop{B_1}\cap\cdots\cap\closedop{B_n}\cap\closedop A)\cup (\cap\closedop{C_1}\cap\cdots\cap\closedop{C_m}\cap\closedop{B})\\
&=\emptyset.
}
This is a contradiction to the finite intersection property.
\item $X\in\sheaff$ by maximality.
\end{enumerate}
Then 
\[
\sheaff\in\bigcap_i\closedop{B_i}
\]
Thus the intersection of the $\closedop{B_i}$ is nonempty.
\end{proof}

\begin{prop}
\label{prop:higson}
If $X$ is a proper metric space then $\corona X$ is homeomorphic to the Higson corona $\nu(X)$ of $X$.
\end{prop}
\begin{proof}
We prove that $\corona X$ is an equivalent definition to $\check X$, the corona of $X$ as defined in \cite[Chapter~2]{Protasov2005} if $X$ is a coarsely proper metric space. Then \cite[Proposition~1]{Protasov2005} implies that $\corona X$ and $\nu(X)$ are homeomorphic if $X$ is a proper metric space. 

Without loss of generality we can assume that $X$ is $R-$discrete for some $R>0$. Then the bounded sets are exactly the finite sets. In this case the non-principal ultrafilters are exactly the cobounded ultrafilters. By Theorem~\ref{thm:cuexists} we can associate to every non-principal ultrafilter a coarse ultrafilter and likewise every coarse ultrafilter on $X$ is induced by a non-principal ultrafilter on $X$ by \cite[Theorem~5.8]{Naimpally1970}. By Proposition~\ref{prop:alikecu} two non-principal ultrafilters are parallel if and only if their induced coarse ultrafilters are asymptotically alike. Finally the topology on $\corona X$ and $\check X$ has been defined in the same way.
\end{proof}

\begin{notat}
If $A,B\s X$ are two subsets of a metric space and $x_0\in X$ a point then define
 \f{
  \chi_{A,B}:\N&\to \R_+\\
  i&\mapsto d(A\ohne B(x_0,i),B\ohne B(x_0,i))
 }
 If $\chi\in \R_+^\N$ is a coarse map then the \emph{class $m(\chi)$ of $\chi$} is at least $f\in \R_+^\N$, written $m(\chi)\ge f$ if
\[
 \chi(i)\ge f(i) +c
\]
where $c\le 0$ is a constant.
Let $\sheaff,\sheafg$ be two coarse ultrafilters on $X$. The distance of $\sheaff$ to $\sheafg$ is at least $f\in\R_+^\N$, written $d(\sheaff,\sheafg)\ge f$ if there are $F\in\sheaff,G\in\sheafg$ with $m(\chi_{F,G})\ge f$.
\end{notat}

\begin{rem}
\label{rem:metrizable}
Note that by \cite[Proposition~8.1]{Dydak2018} there is no countable subset $S\s\R_+^\N$ with the property that for every $f\in\R_+^\N$ there is an element $s\in S$ with $m(f)\ge s$.
\end{rem}

\begin{prop}
\label{prop:countablebase}
 If $X$ is a metric space then the topology on $\corona X$ is coarser than the topology $\tau_d$ induced by $d$.
\end{prop}
\begin{proof}
 Every point $\sheaff\in\corona X$ has a base of neighborhoods $(\closedop A ^c)_A$ where the index $A$ runs over subsets of $X$ with $A\notclose B$ for some $B\in\sheaff$. The topology $\tau_d$ consists of neighborhoods of $\sheaff$ which are finer then the base of neighborhoods.  
 \end{proof}

\section{Coarse Cohomology with twisted Coefficients}

Coarse cohomology with twisted coefficients has been introduced in \cite{Hartmann2017a}. Coarse covers determine a Grothendieck topology on a coarse space and a coarse map gives rise to a morphism of Grothendieck topologies. The resulting sheaf cohomology on coarse spaces is called \emph{coarse cohomology with twisted coefficients}.

\begin{defn}
Let $A\s X$ be a subset of a metric space. Define $\closedop A\s \corona X$ as those classes of coarse ultrafilters for which there is one representative $\sheaff\in\closedop A$.
\end{defn}

\begin{lem}
\label{lem:coarsecovertopology}
If $X$ is a metric space
\begin{itemize}
\item then $(\closedop A^c)_{A\s X}$ are a base for the topology induced by $\close$ on $\corona X$.
\item If $(U_i)_i$ is a coarse cover of $X$ then $(\closedop {U_i^c}^c)_i$ is an open cover of $\corona X$.
\item For every open cover $(V_i)_i$ of $\corona X$ there is a coarse cover $(U_i)_i$ of $X$ such that $(\closedop {U_i^c}^c)_i$ refines $(V_i)_i$.
\item If $(U_i)_i$ is a coarse cover of $X$ then $\closedop {U_i^c}^c\cap\closedop{U_j^c}^c=\closedop{(U_i\cap U_j)^c}^c$ for every $ij$.
\end{itemize}
\end{lem}
\begin{proof}
\begin{itemize}
\item  First we show that if $A\s X$ is a subset then $\closedop A$ is closed. Now the topology induced the proximity relation $\close$ on $\corona X$ is defined by the Kuratowski closure operator
\[
\pi\mapsto \{\sheaff:\sheaff\close \pi\}.
\]
Let $\sheaff\in \corona X$ be an element with $\sheaff\close \closedop A$. Then $B\in\sheaff$ implies $B\close A$. Thus $A\in\sheafg$ with $\sheafg\lambda \sheaff$. Thus $\sheaff\in\closedop A$.

We show $(\closedop S^c)_{S\s X}$ are a base for the topology of $\corona X$. If $A\s \corona X$ is closed and $\sheaff\in A^c$ then $\sheaff\notclose A$. Thus there are $S,T\s X$ with $S\in\sheaff, A\s \closedop T$ such that $S\notclose T$. But then $\sheaff\notclose \closedop T$ thus $\sheaff\in\closedop T^c\s A^c$.
\item We proceed by induction on the number of components of $(U_i)_i$.
\begin{enumerate}
\item If $U$ coarsely covers $X$ then $U^c$ is bounded. Thus $\closedop {U^c}=\emptyset$ which implies $\closedop {U^c}^c=\corona X$.
\item If $U_1^c\notclose U_2^c$ then $\closedop {U_1^c}\notclose\closedop {U_2^c}$ which implies $\closedop {U_1^c}\cap\closedop{U_2^c}=\emptyset$. Thus $\closedop {U_1^c}^c\cup\closedop {U_2^c}^c=\corona X$.
\item If $U,V,U_1,\ldots,U_n$ coarsely covers $X$ then $U\cup V,U_1,\ldots,U_n$ coarsely cover $X$ and $U,V$ coarsely cover $U\cup V$. By induction hypothesis $\closedop {(U\cup V)^c}^c,\closedop{U_1^c}^c,\ldots,\closedop{U_n^c}^c$ is an open cover of $\corona X$. Now $\closedop {(U\cup V)^c}^c\s \closedop{U\cup V}$ and $\closedop{U^c}^c,\closedop{V^c}^c$ are an open cover of $\closedop {U\cup V}$. As a result $\closedop{U^c}^c,\closedop{V^c}^c,\closedop{U_1^c}^c,\ldots,\closedop{U_n^c}^c$ are an open cover of $\corona X$.
\end{enumerate}
\item We can refine $(V_i)_i$ by $(\closedop {U_i^c}^c)_i$ for some $U_i\s X$. By Thereom~\ref{thm:compact} we can refine $(\closedop {U_i^c}^c)_i$ by a finite subcover $\closedop {U_1^c}^c,\ldots,\closedop {U_n^c}^c$. Now
\[
\closedop{U_1^c}\cap\ldots\cap \closedop{U_n^c}=\emptyset
\]
implies $\bigcap_i E[U_i^c]$ is bounded for every entourage $E\s X^2$. This is equivalent to $(U_i)_i$ being a coarse cover.
\item Since $U_i^c\cup U_j^c\z U_i^c,U_j^c$ the inclusion $\closedop{(U_i\cap U_j)^c}^c\s\closedop{U_i^c}^c\cap\closedop{U_j^c}^c$ is obvious. For the reverse inclusion note if $\sheaff$ is a coarse ultrafilter and $U_i^c,U_j^c\not\in\sheaff$ then $U_i^c\cup U_j^c\not\in\sheaff$.
\end{itemize}
\end{proof}

\begin{cor}
\label{cor:cor1}
The functor $\nu'$ preserves and reflects finite coproducts.
\end{cor}
\begin{proof}
In case $X$ is a proper metric space there is a homeomorphism $\corona X=\nu(X)$. Then \cite[Proposition~4.5, Theorem~4.6]{Weighill2016} already states this result. Now Lemma~\ref{lem:coarsecovertopology} serves an alternative proof: A coarse disjoint union\footnote{A coarse disjoint union $(U_i)_i$ is a coarse cover with every two elements $U_i,U_j$ coarsely disjoint.} $(U_i)_i$ of $X$ is mapped to a union of topological connection components $(\corona {U_i})_i$. If for two subsets $A,B\s X$ the closed sets $\closedop A\notclose \closedop B$ are disjoint then $A,B$ are coarsely disjoint.
\end{proof}

\begin{thm}
\label{thm:sheaf}
Let $X$ be a metric space.
\begin{itemize}
\item If $\sheaff$ is a sheaf on $X$ then
\[
\sheaff(\closedop{A^c}^c)=\sheaff(A)
\]
for every $A\s X$ with restriction maps
\[
\sheaff(\closedop{B^c}^c)\to \sheaff(\closedop{A})
\]
equal to $\sheaff(B)\to\sheaff(A)$ for every $A\s B$ determines a sheaf on $\corona X$.
\item If $\sheaff$ is a sheaf on $\corona X$ then 
\[
\sheaff(A)=\sheaff(\closedop{A^c}^c)
\]
for every $A$ with restriction maps
\[
\sheaff(B)\to \sheaff(A)
\]
equal to the restriction map $\sheaff(\closedop{B^c}^c)\to \sheaff(\closedop{A^c}^c)$ for every $A\s B$ is a sheaf on $X$ as a coarse space.
\end{itemize}
\end{thm}
\begin{proof}
\begin{itemize}
\item Now the $\closedop{A^c}^c$ with $A\s X$ are a base for the topology on $\corona X$. By \cite[Chapter~2.7]{Vakil2015} we only need to check the base identity axiom and the base gluability axiom.
\begin{enumerate}
\item base identity axiom: If $(U_i)_i$ is a coarse cover of $U\s X$ and if $x,y\in\sheaff(\closedop{U^c}^c)$ are such that $x|_{\closedop{U_i^c}^c}=y|_{\closedop{U_i^c}^c}$ for every $i$ then for the same $x,y\in\sheaff(U)$ we have $x|_{U_i}=y|_{U_i}$ for every $i$. Thus $x=y$ on $U$. This implies $x=y$ on $\closedop {U^c}^c$.
\item base gluability axiom: If $(U_i)_i$ is a coarse cover of $U\s X$ and we have $x_i\in\sheaff(\closedop{U_i^c}^c)$ such that $x_i|_{\closedop{U_j^c}^c\cap\closedop{U_i^c}^c}=x_j|_{\closedop{U_i^c}^c\cap\closedop{U_j^c}^c}$ for every $i,j$ then the same $x_i\in \sheaff(U_i)$ glue as $x_i|_{U_j}=x_j|_{U_i}$ for every $ij$. This implies there is a global section $x\in \sheaff(U)$ which restricts to $x_i\in\sheaff(U_i)$ for every $i$. Thus the same $x\in\sheaff(\closedop{U^c}^c)$ restricts to $x_i\in\sheaff(\closedop{U_i^c}^c)$ for every $i$.
\end{enumerate}
\item easy.
\end{itemize}
\end{proof}

\begin{cor}
\label{cor:cor2}
Let $X$ be a metric space. If $\sheaff$ is a sheaf on $X$ then
\[
\cohomology i X \sheaff = \cohomology i {\corona X} \sheaff
\]
here the left side denotes sheaf cohomology on coarse spaces and the right side denotes sheaf cohomology on a topological space. Likewise if $\sheaff$ is a sheaf on $\corona X$ the same statement holds. 
\end{cor}

\begin{cor}
\label{cor:cor3}
If two coarse maps between metric spaces $f,g:X\to Y$ are close then they induce isomorphic maps in coarse cohomology with twisted coefficients. 
\end{cor}
\begin{proof}
This has already been proved in \cite[Theorem~72]{Hartmann2017a}. Now Theorem~\ref{thm:sheaf} serves another proof: A coarse map $f:X\to Y$ between coarse spaces gives rise to a morphism of Grothendieck topologies $Y_{ct},X_{ct}$ in the same way as $\corona f$ gives rise to a morphism of Grothendieck topologies of the topological spaces $\corona Y,\corona X$. Thus the direct image functors $f_*$ and $\corona f_*$ are basically the same maps on sheaf level. If $f,g$ are close coarse maps then $\corona f=\corona g$ by Lemma~\ref{lem:closesamemap}, Lemma~\ref{lem:fE(f)}. Thus the result follows.
\end{proof}

\begin{cor}
\label{cor:cor4}
Let $X$ be a metric space. If $\sheaff$ is a sheaf on $X$ and $A,B\s X$ are two subsets that coarsely cover $X$ there is a Mayer-Vietoris long exact sequence in cohomology:
 \f{
 \cdots&\to \cohomology {i-1} {A\cap B}\sheaff\to \cohomology {i} {A\cup B}\sheaff\to\cohomology i A\sheaff\times\cohomology i B \sheaff\\
 &\to \cohomology i {A\cap B}\sheaff\to\cdots
 }
\end{cor}
\begin{proof}
This is already \cite[Theorem~74]{Hartmann2017a}. Now Theorem~\ref{thm:sheaf} gives rise to an alternative proof: If $A,B\s X$ are two subsets of a metric space then $\corona A,\corona B$ can be realized as closed subsets of $\corona X$ by Lemma~\ref{lem:closedembedding}. By \cite[II.Mayer-Vietoris~sequence~5.6]{Iversen1984} there is a long exact Mayer-Vietoris sequence for every sheaf $\sheaff$ on $X$.
\end{proof}

\begin{cor}
\label{cor:cor5}
If $X$ is a proper metric space and the asymptotic dimension $n=\asdim(X)$ of $X$ is finite then
\[
\cohomology q X \sheaff=0
\]
for $q>n$ and $\sheaff$ a sheaf on $X$.
\end{cor}
\begin{proof}
Note that $\corona X$ is paracompact since $\corona X$ is compact. By \cite[Chapitre~II.5.12]{Godement1958} it is sufficient to show that the covering dimension of $\corona X$ does not exceed $n$. Since \cite[Theorem~1.1]{Dranishnikov1998} showed that $\dim(\nu X)\le \asdim(X)$ and Proposition~\ref{prop:higson} showed $\corona X=\nu X$ are homeomorphic the result follows.
\end{proof}

\section{On Morphisms}

\begin{lem}
\label{lem:closedembedding}
If $f:X\to Y$ is a coarsely injective coarse map between metric spaces then $\corona f$ is a closed embedding.
\end{lem}
\begin{proof}
We show if $i:Z\to X$ is an inclusion of metric spaces then $\corona i$ is a closed embedding.

First we prove $\corona i$ is injective: Let $\sheaff,\sheafg$ be two coarse ultrafilters on $Z$. If $i_*\sheaff\lambda i_*\sheafg$ then $\sheaff\lambda\sheafg$ obviously.

Now we prove $\corona i (\corona Z)=\closedop Z$. If $\sheaff$ is a coarse ultrafilter on $X$ with $Z\in\sheaff$ then define
\[
\sheaff|_Z:=\{A\s Z: A\in\sheaff\}.
\]
Then $\sheaff|_Z$ is obviously a coarse ultrafilter.

Now we show $i_*\sheaff|_Z\lambda \sheaff$: if $A\in i_*\sheaff|_Z,B\in\sheaff$ then $A\cap Z\in\sheaff$. Thus $A\close B$.
\end{proof}

\begin{thm}
\label{thm:surjectivecoarselysurjective}
If $f:X\to Y$ is a coarse map between coarsely proper metric spaces and if $\corona f$ is surjective then $f$ is coarsely surjective.
\end{thm}
\begin{proof}
 If $f:X\to Y$ is not coarsely surjective then there exists a nonbounded subspace $Z\s Y$ such that $f(X)\notclose Z$. By Theorem~\ref{thm:cuexists} there exists a coarse ultrafilter $\sheaff$ on $Y$ with $Z\in\sheaff$. This implies $f(X)\not\in\sheaff$. Thus $\sheaff\not\in \im \corona f$.
\end{proof}

\begin{lem}
\label{lem:notclosecoarselyinjective}
A coarse map $f:X\to Y$ between metric spaces is coarsely injective if for every two subsets $A,B\s X$ the relation $A\not\close B$ implies $f(A)\notclose f(B)$.  
\end{lem}
\begin{proof}
The proof is similar to the proof of \cite[Theorem~2.3]{Kalantari2016}. Let $f$ have the above property. Then for every two subsets $A,B\s X$ the relation $A\bar\lambda B$ in $X$ implies $f(A)\bar \lambda f(B)$ in $Y$. Assume the opposite, $f$ is not coarsely injective. Then there is some $r\ge 0$ and a sequence $(x_n,y_n)_n\s X^2$ such that $d(f(x_n),f(y_n))<r$ and $d(x_n,y_n)>n$. Now the sequences $(x_n)_n,(y_n)_n$ satisfy the hypothesis of \cite[Lemma~2.2]{Kalantari2016} which leads to a contradiction. 
\end{proof}

\section{Side Notes}
\begin{rem}\name{Space of Ends}
Let $X$ be a metric space. The relation $\sim$ on $\corona X$ which is defined by belonging to the same topological connection component is an equivalence relation. Similarly as in \cite[Side Notes]{Hartmann2017b} we obtain the space of ends by Freudenthal of $X$, if we assume $X$ to be proper geodesic.
\end{rem}

\begin{rem}
It has been shown in \cite{Kalantari2016} that the Higson corona $\nu X$ of a proper metric space $X$ arises as the boundary of the Smirnov compactification of the proximity space $(X,\delta_\lambda)$. Then \cite[Theorem~7.7]{Naimpally1970} implies that $A\delta_\lambda B$ if and only if $A\cap B\not=\emptyset$ or $(\bar A\cap \nu X)\cap (\bar B\cap \nu X)\not=\emptyset$ for every subsets $A,B\s X$. Thus if $A\cap B=\emptyset$ then $A\notclose B$ if and only if $(\bar A\cap \nu X)\cap(\bar B\cap \nu X)=\emptyset$.
\end{rem}

\begin{lem}
Let $X$ be a metric space. A finite family of sets $(U_i)_i$ is a p-cover of $(X,\delta_\lambda)$ if and only if $(U_i)_i$ is a coarse cover of $X$ such that $\bigcup_i U_i=X$.
\end{lem}
\begin{proof}
Suppose $(V_i)_i$ is a cover of $X$ such that $V_i\bar\delta_\lambda U_i^c$ for every $i$. This implies $V_i\notclose U_i^c$ for every $i$. By Lemma~\ref{lem:pcover} the $(U_i)_i$ are a coarse cover. Furthermore
\f{
\bigcap_i U_i^c
&=\bigcap_i U_i^c\cap \bigcup_i V_i\\
&=\bigcup_i(V_i\cap\bigcap_j U_j^c)\\
&=\emptyset.
}
Thus $\bigcup_i U_i=X$.

Now suppose $(U_i)_i$ are a coarser cover of $X$ with $\bigcup_i U_i=X$. We proceed by induction on the number of components of $(U_i)_i$:
\begin{itemize}
\item $n=1$: The set $X$ coarsely covers $X$.
\item $n=2$: Suppose $U_1,U_2$ coarsely cover $X$ with $U_1\cup U_2=X$. By Lemma~\ref{lem:pcover} we can choose $V_1',V_2'\s X$ with $V_1'\cup V_2'=X$ and $V_1'\notclose U_1^c,V_2'\notclose U_2^c$. Then choose
\[
V_1:=(V_1'\cap U_1)\cup (V_2'\cap U_2^c),V_2:=(V_2'\cap U_2)\cup (V_1'\cap U_1^c)
\]
Then $V_1,V_2$ are a cover of $X$ with $V_1\notclose U_1^c,V_2\notclose U_2^c$ and since $U_1^c\cap U_2^c=\emptyset$ we have
\[
V_1\cap U_1^c=\emptyset, V_2\cap U_2^c=\emptyset
\]
\item $n+1\to n+2$: Suppose $U,V,U_1,\ldots,U_n$ coarsely cover $X$. Then $U,V$ coarsely cover $U\cup V$ and $U\cup V,U_1,\ldots,U_n$ coarsely cover $X$. By induction hypothesis we have subsets $V_1',V_2'\s U\cup V$ with $V_1'\cup V_2'=U\cup V$ and $V_1'\bar\delta_\lambda U^c\cap V,V_2'\bar\delta_\lambda V^c\cap U$. And there are subsets $W,V_1,\ldots,V_n$ with $W\bar\delta_\lambda(U\cup V)^c, V_i\bar\delta_\lambda U_i^c$ for every $i$. Then
\[
V_1'\cap W\bar\delta_\lambda U^c,V_2'\cap W\bar\delta_\lambda V^c.
\]
Then
\[
V_1'\cap W,V_2'\cap W,V_1,\ldots,V_n
\]
is a cover of $X$ with the desired properties.
\end{itemize}
\end{proof}

\bibliographystyle{plain}
\bibliography{mybib}

\address

\end{document}